\documentclass[english,11pt]{article}
\usepackage[latin9]{inputenc}
\usepackage{geometry}

\usepackage{amsmath}
\usepackage{amsthm}
\usepackage{amssymb}
\usepackage{url} 
\usepackage[colorlinks,linkcolor=magenta,citecolor=blue, pagebackref=true]{hyperref}
\renewcommand*{\backrefalt}[4]{%
	\ifcase #1 \footnotesize{(Not cited.)}%
	\or        \footnotesize{(Cited on page~#2.)}%
	\else      \footnotesize{(Cited on pages~#2.)}%
	\fi}
\usepackage{setspace}
\usepackage[authoryear]{natbib}

\makeatletter
\theoremstyle{plain}
\newtheorem{thm}{\protect\theoremname}
\newenvironment{lyxlist}[1]
	{\begin{list}{}
		{\settowidth{\labelwidth}{#1}
		 \setlength{\leftmargin}{\labelwidth}
		 \addtolength{\leftmargin}{\labelsep}
		 }}
	{\end{list}}
\theoremstyle{plain}
\newtheorem{lem}{\protect\lemmaname}
\theoremstyle{plain}
\newtheorem{cor}{\protect\corollaryname}

\makeatother

\usepackage{babel}
\providecommand{\corollaryname}{Corollary}
\providecommand{\lemmaname}{Lemma}
\providecommand{\theoremname}{Theorem}

\newcommand{\lyxaddress}[1]{
	\par {\raggedright #1
		\vspace{1.4em}
		\noindent\par}
}

\providecommand{\keywords}[1]
{
	\small	
	\textbf{\textit{Keywords:}} #1
}

\begin{document}
\title{Approximation of probability density functions\\ via location-scale finite mixtures in Lebesgue spaces \footnote{To appear in {\bf \href{https://www.tandfonline.com/toc/lsta20/current}{Communications in Statistics - Theory and Methods}}.}}
\author{TrungTin Nguyen$^{1**}$,
\and Faicel Chamroukhi$^{1,2}$,
\and Hien D Nguyen$^{2}$, 
\and and Geoffrey J McLachlan$^{3}$}
\date{}
\maketitle

\lyxaddress{
$^{1}$Normandie Univ, UNICAEN, CNRS, LMNO, 14000 Caen, France.\\
$^{2}$School of Engineering and Mathematical Sciences. Department of Mathematics and Statistics, La Trobe University, Melbourne, Victoria, Australia.\\ 
$^{3}$School of Mathematics and Physics, University of Queensland, St. Lucia, Brisbane, Australia.\\
$^{**}$Corresponding author.
}

\begin{abstract}
The class of location-scale finite mixtures is of enduring interest both from applied and theoretical
perspectives of probability and statistics. We establish and prove the following results: to an arbitrary degree
of accuracy, (a) location-scale mixtures of a continuous probability
density function (PDF) can approximate any continuous PDF, uniformly,
on a compact set; and (b) for any finite $p\ge1$, location-scale mixtures
of an essentially bounded PDF can approximate any PDF in $\mathcal{L}_{p}$,
in the $\mathcal{L}_{p}$ norm.
\end{abstract}


\section{Introduction}

Define $\left(\mathbb{E},\left\Vert \cdot\right\Vert _{\mathbb{E}}\right)$
to be a normed vector space (NVS), and let $x\in\left(\mathbb{R}^{n},\left\Vert \cdot\right\Vert _{2}\right)$,
for some $n\in\mathbb{N}$, where $\left\Vert \cdot\right\Vert _{2}$
is the Euclidean norm. Let $f:\mathbb{R}^{n}\rightarrow\mathbb{R}$
be a function satisfying $f\ge0$ and $\int f\text{d}\lambda=1$,
where $\lambda$ is the Lebesgue measure. We say that $f$ is a probability
density function (PDF) on the domain $\mathbb{R}^{n}$ (which we will
omit for brevity, from hereon in). Let $g:\mathbb{R}^{n}\rightarrow\mathbb{R}$
be another PDF and define the functional class $\mathcal{M}^{g}=\bigcup_{m\in\mathbb{N}}\mathcal{M}_{m}^{g}$,
where
\[
\mathcal{M}^{g}_m=\left\{ h_{m}^{g}:h_{m}^{g}\left(\cdot\right)=\sum_{i=1}^{m}\frac{c_{i}}{\sigma_{i}^{n}}g\left(\frac{\cdot-\mu_{i}}{\sigma_{i}}\right),\mu_{i}\in\mathbb{\mathbb{R}}^{n},\sigma_{i}\in\mathbb{R}_{+},c\in\mathbb{S}^{m-1},i\in\left[m\right]\right\} \text{,}
\]
$c^{\top}=\left(c_{1},\dots,c_{m}\right)$, $\mathbb{R}_{+}=\left(0,\infty\right)$,
\[
\mathbb{S}^{m-1}=\left\{ c\in\mathbb{R}^{m}:\sum_{i=1}^{m}c_{i}=1,c_{i}\ge0,i\in\left[m\right]\right\} \text{,}
\]
$\left[m\right]=\left\{ 1,\dots,m\right\} $, and $\left(\cdot\right)^{\top}$
is the matrix transposition operator. We say that $h_{m}^{g}\in\mathcal{M}^{g}$
is an $m\text{-component}$ location-scale finite mixture of the PDF
$g$. The class $\mathcal{M}^{g}$ has enjoyed enduring practical
and theoretical interest throughout the years, as reported in the
volumes of \citet{Everitt:1981aa}, \citet{mclachlan1988mixture}, \citet{Lindsay:1995aa},
\citet{McLachlan:2000aa}, \citet{fruhwirth2006finite}, \citet{Mengersen:2011aa}, \citet{fruhwirth2019handbook}, and \citet{nguyen2021approximationMoE}.

We say that $f$ is compactly supported on $\mathbb{K}\subset\mathbb{R}^{n}$,
if $\mathbb{K}$ is compact and if $\mathbf{1}_{\mathbb{K}^{\complement}}f=0,$
where $\mathbf{1}_{\mathbb{X}}$ is the indicator function that takes
value 1 when $x\in\mathbb{X}$, and $0$ elsewhere, and where $\left(\cdot\right)^{\complement}$
is the set complement operator (i.e. $\mathbb{X}^{\complement}=\mathbb{R}^{n}\backslash\mathbb{X}$).
Here, $\mathbb{X}$ is a generic subset of $\mathbb{R}^{n}$. Further,
say that $f\in\mathcal{L}_{p}\left(\mathbb{X}\right)$ for any $1\le p<\infty$,
if 
\[
\left\Vert f\right\Vert _{\mathcal{L}_{p}\left(\mathbb{X}\right)}=\left(\int\left|\mathbf{1}_{\mathbb{X}}f\right|^{p}\text{d}\lambda\right)^{1/p}<\infty\text{,}
\]
and say that $f\in\mathcal{L}_{\infty}\left(\mathbb{X}\right)$, the class of essentially bounded measurable functions, if 
\[	
\left\Vert f\right\Vert _{\mathcal{L}_{\infty}\left(\mathbb{X}\right)}=\inf\big\{ a\ge0:\lambda\left(\left\{ x\in\mathbb{X}:\left|f\left(x\right)\right|>a\right\} \right)=0\big\} <\infty\text{,}
\]
where we call $\left\Vert \cdot\right\Vert _{\mathcal{L}_{p}\left(\mathbb{X}\right)}$
the $\mathcal{L}_{p}\text{-norm}$ on $\mathbb{X}$. Denote the class of all bounded functions on $\mathbb{X}$ by
\[
\mathcal{B}\left(\mathbb{X}\right)=\left\{ f\in\mathcal{L}_{\infty}\left(\mathbb{X}\right):\exists a\in\left[0,\infty\right)\text{, such that }\left|f\left(x\right)\right|\le a,\forall x\in\mathbb{X}\right\} 
\]
 and write 
\[
\left\Vert f\right\Vert _{\mathcal{B}\left(\mathbb{X}\right)}=\sup_{x\in\mathbb{X}}\left|f\left(x\right)\right|\text{.}
\]
For brevity, we shall write $\mathcal{L}_{p}\left(\mathbb{R}^{n}\right)=\mathcal{L}_{p}$,
$\mathcal{B}\left(\mathbb{R}^{n}\right)=\mathcal{B}$, $\left\Vert f\right\Vert _{\mathcal{L}_{p}\left(\mathbb{R}^{n}\right)}=\left\Vert f\right\Vert _{\mathcal{L}_{p}}$,
and $\left\Vert f\right\Vert _{\mathcal{B}\left(\mathbb{R}^{n}\right)}=\left\Vert f\right\Vert _{\mathcal{B}}$.

Lastly, we denote the class of continuous functions and uniformly
continuous functions by $\mathcal{C}$ and $\mathcal{C}^{u}$, respectively.
The classes of bounded continuous function shall be denoted by $\mathcal{C}_{b}=\mathcal{C}\cap\mathcal{B}$. Note that the class of continuous functions that vanish at infinity, defined as

\[
\mathcal{C}_{0}=\left\{ f\in\mathcal{C}:\forall\epsilon>0,\exists\text{ a compact }\mathbb{K}\subset\mathbb{R}^{n}\text{, such that }\left\Vert f\right\Vert _{\mathcal{B}\left(\mathbb{K}^{\complement}\right)}<\epsilon\right\} \text{,}
\]
is a subset of $\mathcal{C}_{b}$.

An important characteristic of the class $\mathcal{M}^{g}$ is its
capability of approximating larger classes of PDFs in various ways. Motivated by the incomplete proofs of \citet[Lem 3.1]{Xu:1993aa}
and Theorem 5 from \citet[Chapter 20]{Cheney:2000aa}, as well as the
results of \citet{Nestoridis:2007aa}, \citet{Bacharoglou:2010aa},
and \citet{nestoridis2011universal}, \citet{nguyen2020approximation} established and proved the
following theorem regarding sequences of PDFs $\left\{ h_{m}^{g}\right\} $
from $\mathcal{M}^{g}$.
\begin{thm}
[Theorem 5 from \cite{nguyen2020approximation}]\label{thm last paper}Let $h_{m}^{g}\in\mathcal{M}^{g}$
denote an $m\text{-component}$ location finite mixture PDF. If we
assume that $f$ and $g$ are PDFs and that $g\in\mathcal{C}_{0}$,
then the following statements are true.
\begin{lyxlist}{00.00.0000}
\item [{(a)}] For any $f\in\mathcal{C}_{0}$, there exists a sequence $\left\{ h_{m}^{g}\right\} _{m=1}^{\infty}\subset\mathcal{M}^{g}$,
such that
\[
\lim_{m\rightarrow\infty}\left\Vert f-h_{m}^{g}\right\Vert _{\mathcal{L}_{\infty}}=0\text{.}
\]
\item [{(b)}] For any $f\in\mathcal{C}_{b}$, and compact set $\mathbb{K}\subset\mathbb{R}^{n}$,
there exists a sequence $\left\{ h_{m}^{g}\right\} _{m=1}^{\infty}\subset\mathcal{M}^{g}$,
such that 
\[
\lim_{m\rightarrow\infty}\left\Vert f-h_{m}^{g}\right\Vert _{\mathcal{L}_{\infty}\left(\mathbb{K}\right)}=0\text{.}
\]
\item [{(c)}] For any $p\in\left(1,\infty\right)$ and $f\in\mathcal{L}_{p}$,
there exists a sequence $\left\{ h_{m}^{g}\right\} _{m=1}^{\infty}\subset\mathcal{M}^{g}$,
such that 
\[
\lim_{m\rightarrow\infty}\left\Vert f-h_{m}^{g}\right\Vert _{\mathcal{L}_{p}}=0\text{.}
\]
\item [{(d)}] For any measurable $f$, there exists a sequence $\left\{ h_{m}^{g}\right\} _{m=1}^{\infty}\subset\mathcal{M}^{g}$,
such that
\[
\lim_{m\rightarrow\infty}h_{m}^{g}=f\text{, almost everywhere.}
\]
\item [{(e)}] If $\nu$ is a $\sigma\text{-finite}$ Borel measure on $\mathbb{R}^{n}$,
then for any $\nu\text{-measurable }$ $f$, there exists a sequence
$\left\{ h_{m}^{g}\right\} \mathcal{M}^{g}$, such that
\[
\lim_{m\rightarrow\infty}h_{m}^{g}=f\text{, almost everywhere, with respect to }\nu\text{.}
\]
\end{lyxlist}
Further, if we assume that 
\[
g\in\left\{ g\in\mathcal{C}_{0}:\forall x\in\mathbb{R}^{n}\text{, }\left|g\left(x\right)\right|\le\theta_{1}\left(1+\left\Vert x\right\Vert _{2}\right)^{-n-\theta_{2}}\text{, }\left(\theta_{1},\theta_{2}\right)\in\mathbb{R}_{+}^{2}\right\} \text{,}
\]
then the following is also true.
\begin{lyxlist}{00.00.0000}
\item [{(f)}] For any $f\in\mathcal{C}$, there exists a sequence $\left\{ h_{m}^{g}\right\} _{m=1}^{\infty}\subset\mathcal{M}^{g}$,
such that
\[
\lim_{m\rightarrow\infty}\left\Vert f-h_{m}^{g}\right\Vert _{\mathcal{L}_{1}}=0\text{.}
\]
\end{lyxlist}
\end{thm}
The goal of this work is to seek the weakest set of assumptions in order to establish approximation theoretical results over the widest class of probability density problems, possible. 
In this paper, we establish Theorem \ref{thm main res} which improves upon Theorem \ref{thm last paper} in a
number of ways. More specifically, while statements (a), (c), (d), and (e) still hold under the same assumptions as in Theorem \ref{thm last paper}; statement (b) from Theorem \ref{thm last paper} is improved to apply to a larger class of target function $f\in\mathcal{C}$, see more in statement (a) of Theorem \ref{thm main res}; and statement (f) from Theorem \ref{thm last paper} is drastically
improved to apply to any $f\in\mathcal{L}_{1}$ and $g\in\mathcal{L}_{\infty}$, see more in statement (b) of Theorem \ref{thm main res}.
We note in particular that our improvement with respect to statement
(b) from Theorem \ref{thm last paper} yields exactly the result of Theorem 5 from \citet[Chapter 20]{Cheney:2000aa},
which was incorrectly proved (see also \citet[Theorem 33.2]{DasGupta:2008aa}). 

The remainder of the article progresses as follows.
The main result of this paper is stated in Section \ref{sec.mainResult}. Technical
preliminaries to the proof of the main result are presented in Section \ref{sec.technicalPreliminaries}. The proof is then established in Section \ref{sec.proofs}. Additional technical results
required throughout the paper are reported in the Appendix \ref{sec.technicalResults}.

\section{Main result} \label{sec.mainResult}
\begin{thm}
\label{thm main res}Let $h_{m}^{g}\in\mathcal{M}^{g}$ denote an
$m\text{-component}$ location finite mixture PDF. If we assume that
$f$ and $g$ are PDFs, then the following statements are true.
\begin{lyxlist}{00.00.0000}
\item [{(a)}] If $f,g\in\mathcal{C}$ and $\mathbb{K}\subset\mathbb{R}^{n}$
is a compact set, then there exists a sequence $\left\{ h_{m}^{g}\right\} _{m=1}^{\infty}\subset\mathcal{M}^{g}$,
such that 
\[
\lim_{m\rightarrow\infty}\left\Vert f-h_{m}^{g}\right\Vert _{\mathcal{B}\left(\mathbb{K}\right)}=0\text{.}
\]

\item [{(b)}] For $p\in\left[1,\infty\right)$, if $f\in\mathcal{L}_{p}$
and $g\in\mathcal{L}_{\infty}$, then there exists a sequence $\left\{ h_{m}^{g}\right\} _{m=1}^{\infty}\subset\mathcal{M}^{g}$,
such that 
\[
\lim_{m\rightarrow\infty}\left\Vert f-h_{m}^{g}\right\Vert _{\mathcal{L}_{p}}=0\text{.}
\]
\end{lyxlist}
\end{thm}

\section{Technical preliminaries}\label{sec.technicalPreliminaries}

Let $f,g\in\mathcal{L}_{1}$, and denote the convolution of $f$ and
$g$ by $f\star g=g\star f$. Further, we say that $g_{k}\left(\cdot\right)=k^{n}g\left(k\times\cdot\right)$
($k\in\mathbb{R}_{+}$) is a dilate of $g$. 

Notice that $\mathcal{M}_{m}^{g}$ can be parameterized via dilates.
That is, we can write
\[
\mathcal{M}_{m}^{g}=\left\{ h_{m}^{g}:h_{m}^{g}\left(\cdot\right)=\sum_{i=1}^{m}c_{i}k_{i}^{n}g\left(k_{i}\times\cdot-k_{i}\mu_{i}\right),\mu_{i}\in\mathbb{\mathbb{R}}^{n},k_{i}\in\mathbb{R}_{+},c\in\mathbb{S}^{m-1},i\in\left[m\right]\right\} \text{,}
\]
where $k_{i}=1/\sigma_{i}$.

Let $\mathbb{F}$ be a subset of $\mathbb{E}$, and denote the convex
hull of $\mathbb{F}$ by $\text{conv}\left(\mathbb{F}\right)$ is
the smallest convex subset in $\mathbb{E}$ that contains $\mathbb{F}$
(cf. \citealp[Chapter 1]{Brezis:2010aa}). By definition, we may write
\[
\text{conv}\left(\mathbb{F}\right)=\left\{ \sum_{i\in\left[m\right]}\alpha_{i}f_{i}:f_{i}\in\mathbb{F},\alpha\in\mathbb{S}^{m-1},i\in\left[m\right],m\in\mathbb{N}\right\} \text{,}
\]
where $\alpha^{\top}=\left(\alpha_{1},\dots,\alpha_{m}\right)$.

Define the class of ``basic'' densities, which will serve as the approximation building blocks, as follows
\[
\mathcal{G}^{g}=\left\{ k^{n}g\left(k\times\cdot-k\mu\right),\mu\in\mathbb{R}^{n},k\in\mathbb{R}_{+}\right\} \text{,}
\]
and suppose that we can choose a suitable NVS $\left(\mathbb{E},\left\Vert \cdot\right\Vert _{\mathbb{E}}\right)$,
such that $\mathcal{G}^{g}\subset\mathcal{M}^{g}\subset\mathbb{E}$.
Then, by definition, it holds that $\mathcal{M}^{g}$ is a convex hull of $\mathcal{G}^{g}$.

For $u\in\mathbb{E}$ and $r>0$, we define the open and closed balls
of radius $r$, centered around $u$, by:
\[
\mathbb{B}\left(u,r\right)=\left\{ v\in\mathbb{E}:\left\Vert u-v\right\Vert _{\mathbb{E}}<r\right\} \text{,}
\]
and 
\[
\overline{\mathbb{B}}\left(u,r\right)=\left\{ v\in\mathbb{E}:\left\Vert u-v\right\Vert _{\mathbb{E}}\le r\right\} \text{,}
\]
respectively. For brevity, we also write $\mathbb{B}_{r}=\mathbb{B}\left(0,r\right)$
and $\overline{\mathbb{B}}_{r}=\overline{\mathbb{B}}\left(0,r\right)$.
A set $\mathbb{F}\subset\mathbb{E}$ is open, if for every $u\in\mathbb{F}$,
there exists an $r>0$, such that $\mathbb{B}\left(u,r\right)\subset\mathbb{F}$.
We say that $\mathbb{F}$ is closed if its complement is open, and
by definition, we say that $\mathbb{E}$ and the empty set are both
closed and open.

We call the smallest closed set containing $\mathbb{F}$ its closure,
and we denote it by $\overline{\mathbb{F}}$. A sequence $\left\{ u_{m}\right\} \subset\mathbb{E}$
converges to $u\in\mathbb{E}$, if $\lim_{m\rightarrow\infty}\left\Vert u_{m}-u\right\Vert _{\mathbb{E}}=0$,
and we denote it symbolically by $\lim_{m\rightarrow\infty}u_{m}=u$.
That is, for every $\epsilon>0$, there exists an $N\left(\epsilon\right)\in\mathbb{N}$,
such that $m\ge N\left(\epsilon\right)$ implies that $\left\Vert u_{m}-u\right\Vert _{\mathbb{E}}<\epsilon$. 

By Lemma \ref{lem NVS}, we can write the closure of $\mathbb{F}$
as 
\[
\overline{\mathbb{F}}=\left\{ u\in\mathbb{E}:u=\lim_{m\rightarrow\infty}u_{m},u_{m}\in\mathbb{F}\right\} 
\]
and hence
\[
\overline{\mathcal{M}^{g}}=\left\{ h\in\mathbb{E}:h=\lim_{m\rightarrow\infty}h_{m}^{g},h_{m}^{g}\in\mathcal{M}^{g}\right\} \text{.}
\]
Thus, by definition, it holds that $\overline{\mathcal{M}^{g}}$ is
a closed and convex subset of $\mathbb{E}$.

If $f\in\mathcal{C}$ is a PDF on $\mathbb{R}^{n}$, we denote its
support by
\[
\text{supp}f=\left\{ x\in\mathbb{R}^{n}:f\left(x\right)\ne0\right\} 
\]
 and furthermore, we denote the set of compactly supported continuous
functions by
\[
\mathcal{C}_{c}=\left\{ f\in\mathcal{C}:\text{supp}f\text{ is compact}\right\} \text{.}
\]
For open sets $\mathbb{V}\subset\mathbb{R}^{n}$, we will write $f\prec\mathbb{V}$
as shorthand for $f\in\mathcal{C}_{c}$, $0\le f\le1$, and $\text{supp}f\subset\mathbb{V}$.

The following lemmas permit us to prove the primary technical mechanism that is used to prove our main result presented in Theorem \ref{thm main res}.

\begin{lem}
\label{lem compact approx}Let $f\in\mathcal{C}$ be a PDF. Then,
for every compact $\mathbb{K}\subset\mathbb{R}^{n}$, we can choose
$h\in\mathcal{C}_{c}$, such that $\mathrm{supp}\,h\subset\mathbb{B}_{r}$,
$0\le h\le f$, and $h=f$ on $\mathbb{K}$, for some $r\in\mathbb{R}_{+}$.
\end{lem}
\begin{proof}
Since $\mathbb{K}$ is bounded, there exists some $r\in\mathbb{R}_{+}$,
such that $\mathbb{K}\subset\mathbb{B}_{r}$. Lemma \ref{lem part unity}
implies that there exists a function $u\prec\mathbb{B}_{r}$, such
that $u\left(x\right)=1$, for all $x\in\mathbb{K}$. We can then set $h=uf$
to obtain the desired result of Lemma \ref{lem compact approx}.
\end{proof}
\begin{lem}
\label{lem bounded convolution}Let $h\in\mathcal{C}_{c}$, such that
$\mathrm{supp}\,h\subset\mathbb{B}_{r}$, $0\le h$, and $\int h\text{d}\lambda\le1$,
and let $g\in\mathcal{C}$ be a PDF. Then, for any $k\in\mathbb{R}_{+}$,
there exists a sequence $\left\{ h_{m}^{g}\right\} _{m=1}^{\infty}\subset\mathcal{M}^{g}$,
so that
\begin{equation}
\lim_{m\rightarrow\infty}\left\Vert g_{k}\star h-h_{m}^{g}\right\Vert _{\mathcal{B}\left(\overline{\mathbb{B}}_{r}\right)}=0\text{.}\label{eq: convolution to mixture lemma part 1}
\end{equation}
Furthermore, if $g\in\mathcal{C}_{b}^{u}$, we have the stronger result
that
\begin{equation}
\lim_{m\rightarrow\infty}\left\Vert g_{k}\star h-h_{m}^{g}\right\Vert _{\mathcal{B}}=0\text{.}\label{eq: convolution to mixture lemma part 2}
\end{equation}
\end{lem}
\begin{proof}
It suffices to show that given any $r,k,\epsilon\in\mathbb{R}_{+}$,
there exists a sufficiently large $m\left(\epsilon,r,k\right)\in\mathbb{N}$ such that for all $m\ge m\left(\epsilon,r,k\right)$, there exists
a $h_{m}^{g}\in\mathcal{M}_{m}^{g}$ satisfying \begin{equation}
\left\Vert g_{k}\star h-h_{m}^{g}\right\Vert _{\mathcal{B}\left(\overline{\mathbb{B}}_{r}\right)}<\epsilon\text{.}\label{eq: convolution to mixture lemma 1}
\end{equation}

First, write 
\begin{eqnarray*}
\left(g_{k}\star h\right)\left(x\right) & = & \int g_{k}\left(x-y\right)h\left(y\right)\text{d}\lambda\left(y\right)=\int\mathbf{1}_{\left\{ y:y\in\overline{\mathbb{B}}_{r}\right\} }g_{k}\left(x-y\right)h\left(y\right)\text{d}\lambda\left(y\right)\\
 & = & \int\mathbf{1}_{\left\{ y:y\in\overline{\mathbb{B}}_{r}\right\} }k^{n}g\left(kx-ky\right)h\left(y\right)\text{d}\lambda\left(y\right)=\int\mathbf{1}_{\left\{ z:z\in\overline{\mathbb{B}}_{rk}\right\} }g\left(kx-z\right)h\left(\frac{z}{y}\right)\text{d}\lambda\left(z\right)\text{,}
\end{eqnarray*}
where $\overline{\mathbb{B}}_{rk}$ is a continuous image of a compact
set, and hence is also compact (cf. \citealp[Theorem 4.14]{Rudin:1976aa}).
By Lemma \ref{lem ball cover}, for any $\delta>0$, there exist $\kappa_{i}\in\mathbb{R}^{n}$
($i\in\left[m-1\right]$, for some $m\in\mathbb{N}$), such that $\overline{\mathbb{B}}_{rk}\subset\bigcup_{i=1}^{m-1}\mathbb{B}\left(\kappa_{i},\delta/2\right)$.
Further, if $\mathbb{B}_{i}^{\delta}=\mathbb{B}_{rk}^{\delta}=\overline{\mathbb{B}}_{rk}\cap\mathbb{B}\left(\kappa_{i},\delta/2\right)$,
then $\overline{\mathbb{B}}_{rk}=\bigcup_{i=1}^{m-1}\mathbb{B}_{i}^{\delta}$.
We can hence obtain a disjoint covering of $\overline{\mathbb{B}}_{rk}$
by taking $\mathbb{A}_{1}^{\delta}=\mathbb{B}_{1}^{\delta}$, and
$\mathbb{A}_{i}^{\delta}=\mathbb{B}_{i}^{\delta}\backslash\bigcup_{j=1}^{i-1}\mathbb{B}_{j}^{\delta}$
($i\in\left[m-1\right]$) (cf. \citealp[Chapter 24]{Cheney:2000aa}).
Notice that $\overline{\mathbb{B}}_{rk}=\bigcup_{i=1}^{m-1}\mathbb{A}_{i}^{\delta}$,
each $\mathbb{A}_{i}^{\delta}$ is a Borel set, and $\text{diam}\left(\mathbb{A}_{i}^{\delta}\right)\le\delta$,
by construction.

We shall denote the disjoint cover of $\overline{\mathbb{B}}_{rk}$
by $\Pi_{m}^{\delta}=\left\{ \mathbb{A}_{i}^{\delta}\right\} _{i=1}^{m-1}$.
We seek to show that there exists an $m\in\mathbb{N}$ and $\Pi_{m}^{\delta}$,
such that
\[
\left\Vert g_{k}\star h-\sum_{i=1}^{m}c_{i}k_{i}^{n}g\left(k_{i}x-z_{i}\right)\right\Vert _{\mathcal{B}\left(\overline{\mathbb{B}}_{r}\right)}<\epsilon\text{,}
\]
where $k_{i}=k$, $c_{i}=k^{-n}\int\mathbf{1}_{\left\{ z:z\in\mathbb{A}_{i}^{\delta}\right\} }h\left(z/k\right)\text{d}\lambda\left(z\right)$,
and $z_{i}\in\mathbb{A}_{i}^{\delta}$, for $i\in\left[m-1\right]$.
We then set $z_{m}=0$ and $c_{m}=1-\sum_{i=1}^{m-1}c_{i}$. Here,
$c_{m}$ depends only on $r$ and $\epsilon$. Suppose that $c_{m}>0$.
Then, since $g\ne0$, there exists some $s\in\mathbb{R}_{+}$ such
that $C_{s}=\sup_{w\in\overline{\mathbb{B}}_{s}}g\left(w\right)>0$.
We can choose

\[
k_{m}=\min\left\{ \frac{s}{r},\left(\frac{\epsilon}{2c_{m}C_{s}}\right)^{1/n}\right\} \text{,}
\]
so that $\left\Vert g\left(k_{m}\times\cdot\right)\right\Vert _{\mathcal{B}\left(\overline{\mathbb{B}}_{r}\right)}\le C_{s}$
and
\[
\left\Vert g\left(k_{m}\times\cdot\right)\right\Vert _{\mathcal{B}\left(\overline{\mathbb{B}}_{r}\right)}\le\frac{c_{m}\epsilon C_{s}}{2c_{m}C_{s}}=\epsilon/2\text{.}
\]
Moreover, if we assume that $g\in\mathcal{C}_{b}^{u}$, then there
exists a constant $C\in\left(0,\infty\right)$ such that $\left\Vert g\right\Vert _{\mathcal{B}}\le C$.
In this case, we can choose $k_{m}^{n}=\epsilon/\left(2c_{m}C\right)$
to obtain
\[
\left\Vert c_{m}k_{m}^{n}g\left(k_{m}\times\cdot-z_{m}\right)\right\Vert _{\mathcal{B}}\le\epsilon/2\text{.}
\]

Since $0\le h$ and $\int h\text{d}\lambda\in\left[0,1\right]$, the
sum $\sum_{i=1}^{m-1}c_{i}$ satisfies the inequalities:
\begin{eqnarray*}
0\le\sum_{i=1}^{m-1}c_{i} & = & k^{-n}\sum_{i=1}^{m-1}\int\mathbf{1}_{\left\{ z:z\in\mathbb{A}_{i}^{\delta}\right\} }h\left(\frac{z}{k}\right)\text{d}\lambda\left(z\right)\\
 & = & k^{-n}\int\mathbf{1}_{\left\{ z:z\in k\mathbb{K}\right\} }h\left(\frac{z}{k}\right)\text{d}\lambda\left(z\right)=\int\mathbf{1}_{\left\{ x:x\in\mathbb{K}\right\} }h\text{d}\lambda\le1\text{.}
\end{eqnarray*}
Thus, $c_{m}\in\left[0,1\right]$, and our construction of $h_{m}^{g}$ implies that
$h_{m}^{g}= \sum_{i=1}^{m}c_{i}k_{i}^{n}g\left(k_{i}x-z_{i}\right)\in\mathcal{M}_{m}^{g}$.

We can then bound the left-hand side of (\ref{eq: convolution to mixture lemma 1})
as follows:
\begin{align}
&\left\Vert g_{k}\star h-h_{m}^{g}\right\Vert _{\mathcal{B}\left(\overline{\mathbb{B}}_{r}\right)}\nonumber\\
\le & \left\Vert g_{k}\star h-\sum_{i=1}^{m-1}c_{i}k_{i}^{n}g\left(k_{i}\times\cdot-z_{i}\right)\right\Vert _{\mathcal{B}\left(\overline{\mathbb{B}}_{r}\right)}+\left\Vert c_{m}k_{m}^{n}g\left(k_{m}\times\cdot-z_{m}\right)\right\Vert _{\mathcal{B}\left(\overline{\mathbb{B}}_{r}\right)}\nonumber \\
\le & \left\Vert g_{k}\star h-\sum_{i=1}^{m-1}c_{i}k_{i}^{n}g\left(k_{i}\times\cdot-z_{i}\right)\right\Vert _{\mathcal{B}\left(\overline{\mathbb{B}}_{r}\right)}+\frac{\epsilon}{2}\nonumber \\
= & \left\Vert \int\text{\textbf{1}}_{\left\{ z:z\in\overline{\mathbb{B}}_{rk}\right\} }g\left(kx-z\right)h\left(\frac{z}{k}\right)\text{d}\lambda\left(z\right)-\sum_{i=1}^{m-1}\int\text{\textbf{1}}_{\left\{ z:z\in\mathbb{A}_{i}^{\delta}\right\} }g\left(kx-z\right)h\left(\frac{z}{k}\right)\text{d}\lambda\left(z\right)\right\Vert _{\mathcal{B}\left(\overline{\mathbb{B}}_{r}\right)}\nonumber \\ &
\quad +\frac{\epsilon}{2}\nonumber \\
\le & \sum_{i=1}^{m-1}\int\text{\textbf{1}}_{\left\{ z:z\in\mathbb{A}_{i}^{\delta}\right\} }\left|g\left(kx-z\right)-g\left(kx-z_{i}\right)\right|h\left(\frac{z}{k}\right)\text{d}\lambda\left(z\right)+\frac{\epsilon}{2}\text{.}\label{eq: convolution to mixture lemma 2}
\end{align}

Since $x\in\overline{\mathbb{B}}_{r}$, $z\in\mathbb{A}_{i}^{\delta}$,
and $z_{i}\in\overline{\mathbb{B}}_{rk}$, it holds that
$\left\Vert kx-z_{i}\right\Vert _{2}=\left\Vert kx-z\right\Vert _{2}\le2rk\text{,}$
and
\[
\left\Vert kx-z-\left(kx-z_{i}\right)\right\Vert _{2}=\left\Vert z-z_{i}\right\Vert _{2}\le\mathrm{diam}\left(\mathbb{A}_{i}^{\delta}\right)\le\delta\text{.}
\]
Note that $g\in\mathcal{C}$, and thus $g$ is uniformly continuous
on the compact set $\overline{\mathbb{B}}_{2rk}$, implying that

\[
\left|g\left(kx-z\right)-g\left(kx-z_{i}\right)\right|\le w\left(g,2rk,\delta\right)\text{,}
\]
for each $i\in\left[m-1\right]$, where
\[
w\left(g,r,\delta\right)=\sup\left\{ \left|g\left(x\right)-g\left(y\right)\right|:\left\Vert x-y\right\Vert _{2}\le\delta\text{ and }x,y\in\overline{\mathbb{B}}_{r}\right\} 
\]
denotes a modulus of continuity. Since $\lim_{\delta\rightarrow0}w\left(g,2rk,\delta\right)=0$
(cf. \citealp[Theorem 4.7.3]{Makarov:2013aa}), we may choose a $\delta\left(\epsilon,r,k\right)>0$,
such that
\[
w\left(g,2rk,\delta\left(\epsilon,r,k\right)\right)<\frac{\epsilon}{2k^{n}}\text{.}
\]

We then proceed from (\ref{eq: convolution to mixture lemma 2}) as
follows:
\begin{align}
\left\Vert g_{k}\star h-h_{m}^{g}\right\Vert _{\mathcal{B}\left(\overline{\mathbb{B}}_{r}\right)} & \le w\left(g,2rk,\delta\left(\epsilon,r,k\right)\right)\int\mathbf{1}_{\left\{ z:z\in\overline{\mathbb{B}}_{rk}\right\} }h\left(\frac{z}{k}\right)\text{d}\lambda\left(z\right)+\frac{\epsilon}{2}\nonumber \\
 & =w\left(g,2rk,\delta\left(\epsilon,r,k\right)\right)k^{n}\int h\text{d}\lambda+\frac{\epsilon}{2}\nonumber \\
 & \le w\left(g,2rk,\delta\left(\epsilon,r,k\right)\right)k^{n}+\frac{\epsilon}{2}<\frac{\epsilon}{2}+\frac{\epsilon}{2}=\epsilon\text{.}\label{eq: convolution to mixture lemma 3}
\end{align}
To conclude the proof of (\ref{eq: convolution to mixture lemma part 1}),
it suffices to choose an appropriate sequence of partitions $\Pi_{m}^{\delta\left(\epsilon,r,k\right)}$,
such that $m\ge m\left(\epsilon,r,k\right)$, for some sufficiently
large $m\left(\epsilon,r,k\right)$, so that (\ref{eq: convolution to mixture lemma 2})
and (\ref{eq: convolution to mixture lemma 3}) hold. This is possible
via Lemma \ref{lem ball cover}. When $g\in\mathcal{C}_{b}^{u}$, we notice that (\ref{eq: convolution to mixture lemma 2}) and (\ref{eq: convolution to mixture lemma 3})
both hold for all $x\in\mathbb{R}^{n}$. Thus, we have the stronger
result of (\ref{eq: convolution to mixture lemma part 2}).
\end{proof}
We present the primary tools for proving Theorem (\ref{thm main res})
in the following pair of lemma. The first one in Lemma \ref{lem main 1} permits the approximation
of convolutions of the form $g_{k}\star f$ in the $\mathcal{L}_{1}$
functional space, and the second presented in Lemma \ref{lem main p} generalizes this first result to the
spaces $\mathcal{L}_{p}$, where $p\in\left[1,\infty\right)$, under an essentially bounded assumption.
\begin{lem}
\label{lem main 1}If $f$ and $g$ are PDFs in the NVS $\left(\mathcal{L}_{1},\left\Vert \cdot\right\Vert _{\mathcal{L}_{1}}\right)$,
then $\mathcal{M}^{g}\subset\mathcal{L}_{1}$ and $g_{k}\star f\in\mathcal{L}_{1}$,
for every $k\in\mathbb{R}_{+}$. Furthermore, there exists a sequence
$\left\{ h_{m}^{g}\right\} _{m=1}^{\infty}\subset\mathcal{M}^{g}$, such that
\[
\lim_{m\rightarrow\infty}\left\Vert g_{k}\star f-h_{m}^{g}\right\Vert _{\mathcal{L}_{1}}=0\text{.}
\]
\end{lem}
\begin{proof}
For any $k\in\mathbb{R}_{+}$, we can show that $g_{k}\in\mathcal{L}_{1}$,
since
\[
\left\Vert g_{k}\right\Vert _{\mathcal{L}_{1}}=\int g_{k}\text{d}\lambda=\int k^{n}g\left(kx\right)\text{d}\lambda\left(x\right)=\int g\text{d}\lambda=1\text{.}
\]
If $h_{m}^{g}\in\mathcal{M}_{m}^{g}$, then $h_{m}^{g}\in\mathcal{L}_{1}$,
since it is a finite sum of functions in $\mathcal{L}_{1}$, and thus,
$\mathcal{M}^{g}\subset\mathcal{L}_{1}$. Note that since $f$ is
a PDF, we have $f\in\mathcal{L}_{1}$, and by Lemma \ref{lem convolution},
we also have that $g_{k}\star f\in\mathcal{L}_{1}$. By Lemma \ref{lem fubini},
it then follows that
\begin{eqnarray*}
\left\Vert g_{k}\star f\right\Vert _{\mathcal{L}_{1}} & = & \int g_{k}\star f\text{d}\lambda\\
 & = & \int\left[\int g_{k}\left(x-y\right)f\left(y\right)\text{d}\lambda\left(y\right)\right]\text{d}\lambda\left(x\right)\\
 & = & \int\left[\int g_{k}\left(x-y\right)\text{d}\lambda\left(x\right)\right]f\left(y\right)\text{d}\lambda\left(y\right)\\
 & = & \left\Vert g_{k}\right\Vert _{\mathcal{L}_{1}}\left\Vert f\right\Vert _{\mathcal{L}_{1}}=1\text{}
\end{eqnarray*}

By definition of of the closure of $\mathcal{M}^{g}$ in $\mathcal{L}_{1}$,
it suffices to show that for any $k\in\mathbb{R}_{+}$, $g_{k}\star f\in\overline{\mathcal{M}^{g}}$.
We seek a contradiction by assuming that $g_{k}\star f\notin\overline{\mathcal{M}^{g}}$.
Then, we can choose $\mathbb{A}=\overline{\mathcal{M}^{g}}$ and $\mathbb{B}=\left\{ g_{k}\star f\right\} $
so that $\mathbb{A},\mathbb{B}\subset\mathcal{L}_{1}$ are nonempty
convex subsets, such that $\mathbb{A}\cap\mathbb{B}=\emptyset$. Furthermore,
$\mathbb{A}$ is closed and $\mathbb{B}$ is compact. By Lemma \ref{lem hahn banach},
there exists a continuous linear functional $\phi\in\mathcal{L}_{1}^{*}$,
such that $\phi\left(v\right)<\alpha<\phi\left(w\right)$, for all
$v\in\mathbb{A}$ and $w\in\mathbb{B}$. By definition of $\mathbb{B}$,
for all $v\in\overline{\mathcal{M}^{g}}\subset\mathcal{L}_{1}$ we
have
\[
\phi\left(v\right)<\alpha<\phi\left(g_{k}\star f\right)\text{.}
\]

By Lemma \ref{lem Riesz 1}, with $\phi\in\mathcal{L}_{1}^{*}$, there
exists a unique function $u\in\mathcal{L}_{\infty}$, such that, for
all $v\in\mathcal{L}_{1}$,
\[
\phi\left(v\right)=\int u\left(x\right)v\left(x\right)\text{d}\lambda\left(x\right)\text{.}
\]
If we let $v=g_{k}\left(\cdot-\mu\right)\in\overline{\mathcal{M}^{g}}\subset\mathcal{L}_{1}$,
then we obtain the inequalities
\[
\sup_{\mu\in\mathbb{R}^{n}}\int u\left(x\right)g_{k}\left(x-\mu\right)\text{d}\lambda\left(x\right)<\alpha<\int u\left(x\right)\left(g_{k}\star f\right)\left(x\right)\text{d}\lambda\left(x\right)\text{.}
\]
The left-hand inequality can be reduced as follows:
\begin{align*}
\alpha & <\int u\left(x\right)\left(g_{k}\star f\right)\left(x\right)\text{d}\lambda\left(x\right)\\
 & =\int u\left(x\right)\left[\int g_{k}\left(x-\mu\right)f\left(\mu\right)\text{d}\lambda\left(\mu\right)\right]\text{d}\lambda\left(x\right)\\
 & =\int f\left(\mu\right)\left[\int u\left(x\right)g_{k}\left(x-\mu\right)\text{d}\lambda\left(x\right)\right]\text{d}\lambda\left(\mu\right)\\
 & <\alpha\int f\left(\mu\right)\text{d}\lambda\left(\mu\right)=\alpha\text{,}
\end{align*}
where the third line is due to Lemma \ref{lem fubini} and the final
equality is because $f$ is a PDF. This yields the sought contradiction.
\end{proof}
\begin{lem}
\label{lem main p}If $f,g\in\mathcal{L}_{\infty}$ are PDFs in the
NVS $\left(\mathcal{L}_{\infty},\left\Vert \cdot\right\Vert _{\mathcal{L}_{p}}\right)$,
for $p\in\left[1,\infty\right)$, then, $\mathcal{M}^{g}\subset\mathcal{L}_{p}$
and $g_{k}\star f\in\mathcal{L}_{p}$, for any $k\in\mathbb{R}_{+}$.
Furthermore, there exists a sequence $\left\{ h_{m}^{g}\right\} _{m=1}^{\infty}\subset\mathcal{M}^{g}$,
such that
\[
\lim_{m\rightarrow\infty}\left\Vert g_{k}\star f-h_{m}^{g}\right\Vert _{\mathcal{L}_{p}}=0\text{.}
\]
\end{lem}
\begin{proof}
We obtain the result for $p=1$ via Lemma \ref{lem main 1}. Otherwise,
since $g\in\mathcal{L}_{1}\cap\mathcal{L}_{\infty}$, we know that
$g\in\mathcal{L}_{p}$ and $g_{k}\in\mathcal{L}_{p}$, for each $k\in\mathbb{R}_{+}$,
via Lemma \ref{lem inclusion}. For any $h_{m}^{g}\in\mathcal{M}_{m}^{g}$,
we then have $h_{m}^{g}\in\mathcal{L}_{p}$ via finite summation,
and hence $\mathcal{M}^{g}\in\mathcal{L}_{p}$. Since $f\in\mathcal{L}_{1}$,
Lemma \ref{lem convolution} implies that $g_{k}\star f\in\mathcal{L}_{p}$.
By definition of the closure of $\mathcal{M}^{g},$ it suffices to
show that $g_{k}\star f\in\overline{\mathcal{M}^{g}}$, for any $k\in\mathbb{R}_{+}$.
This can be achieved by seeking a contradiction under the assumption
that $g_{k}\star f\notin\overline{\mathcal{M}^{g}}$ and using Lemma
\ref{lem Riesz p} in the same manner as Lemma \ref{lem Riesz 1}
is used in the proof of Lemma \ref{lem main 1}.
\end{proof}

\section{Proof of main result}\label{sec.proofs}

\subsection{Proof of Theorem \ref{thm main res} (a)}
To prove the statement (a) of Theorem \ref{thm main res}, it suffices to show that there exists a sufficiently large $m\left(\epsilon,\mathbb{K}\right)\in\mathbb{N}$,
such that for all $m\ge m\left(\epsilon,\mathbb{K}\right)$, there
exists a $h_{m}^{g}\in\mathcal{M}_{m}^{g}$, such that $\left\Vert f-h_{m}^{g}\right\Vert _{\mathcal{B}\left(\mathbb{K}\right)}<\epsilon$,
for any $\epsilon>0$ and compact set $\mathbb{K}\subset\mathbb{R}^{n}$.

First, Lemma \ref{lem compact approx} implies that we can choose
a $h\in\mathcal{C}_{c}$, such that $\mathrm{supp}\:h\subset\overline{\mathbb{B}}_{r}$,
$0\le h\le f$, and $h=f$ on $\mathbb{K}$, for some $r>0$, where
$\mathbb{K}\subset\overline{\mathbb{B}}_{r}$. We then have $\left\Vert f-h\right\Vert _{\mathcal{B}\left(\mathbb{K}\right)}=0$.

Since $h\in\mathcal{C}_{c}\subset\mathcal{C}_{b}^{u}$, Lemma \ref{lem approximate identities} and Corollary \ref{cor dilation}
then imply that there exists a $k\left(\epsilon\right)\in\mathbb{R}_{+}$,
such that for all $k\ge k\left(\epsilon\right)$, $\left\Vert h-g_{k}\star h\right\Vert _{\mathcal{B\left(\mathbb{K}\right)}}<\epsilon/2$.
We shall assume that $k\ge k\left(\epsilon\right)$, from hereon in.

Lemma \ref{lem bounded convolution} then implies that there exists
an $m\left(\epsilon,r,k\right)\in\mathbb{N}$, such that for any $m\ge m\left(\epsilon,r,k\right)$,
there exists a $h_{m}^{g}\in\mathcal{M}_{m}^{g}$, such that $\left\Vert g_{k}\star h-h_{m}^{g}\right\Vert _{\mathcal{B}\left(\mathbb{K}\right)}<\left\Vert g_{k}\star h-h_{m}^{g}\right\Vert _{\mathcal{B}\left(\overline{\mathbb{B}}_{r}\right)}<\epsilon/2$.
The triangle inequality then completes the proof.

\subsection{Proof of Theorem \ref{thm main res} (b)}

To prove the statement (a) of Theorem \ref{thm main res}, it suffices to show that there exists a sufficiently large $m\left(\epsilon\right)\in\mathbb{N}$,
such that for all $m\ge m\left(\epsilon\right)$, there exists a $h_{m}^{g}\in\mathcal{M}_{m}^{g}$,
such that $\left\Vert f-h_{m}^{g}\right\Vert _{\mathcal{L}_{p}}<\epsilon$,
for any $\epsilon>0$.

First, Lemma \ref{lem approximate identities} and Corollary \ref{cor dilation}
imply that there exists a $k\left(\epsilon\right)\in\mathbb{R}_{+}$,
such that for any $k\ge k\left(\epsilon\right)$, it follows that
$\left\Vert f-g_{k}\star f\right\Vert _{\mathcal{L}_{p}}<\epsilon/2$.
We shall assume $k\ge k\left(\epsilon\right)$, from hereon in.

Lemmas \ref{lem main 1} and \ref{lem main p} imply that there exists
an $m\left(\epsilon\right)\in\mathbb{N}$, such that for all $m\ge m\left(\epsilon\right)$,
there exists a $h_{m}^{g}\in\mathcal{M}_{m}^{g}$, such that $\left\Vert g_{k}\star f-h_{m}^{g}\right\Vert _{\mathcal{L}_{p}}<\epsilon/2$.
The triangle inequality then completes the proof.

\appendix
\section{Technical results} \label{sec.technicalResults}

We state a number of technical results that are used throughout the
main text, in this Appendix. Sources for unproved results are provided
at the end of the section.
\begin{lem}
\label{lem approximate identities}Let $\left\{ g_{k}\right\} $ be
a sequence of PDFs in $\mathcal{L}_{1}$, such that for every $\delta>0$,
\[
\lim_{k\rightarrow\infty}\int\mathbf{1}_{\left\{ x:\left\Vert x\right\Vert _{2}>\delta\right\} }g_{k}\text{d}\lambda=0\text{.}
\]
Then, for $f\in\mathcal{L}_{p}$ and $p\in\left[1,\infty\right)$,
\[
\lim_{k\rightarrow\infty}\left\Vert g_{k}\star f-f\right\Vert _{\mathcal{L}_{p}}=0\text{.}
\]
Furthermore, for $f\in\mathcal{C}_{b}$ and compact $\mathbb{K}\subset\mathbb{R}^{n}$,

\[
\lim_{k\rightarrow\infty}\left\Vert g_{k}\star f-f\right\Vert _{\mathcal{L}_{\infty}\left(\mathbb{K}\right)}=0\text{.}
\]
\end{lem}
The sequences $\left\{ g_{k}\right\} $ of Lemma \ref{lem approximate identities}
are often referred to as approximate identities or approximations
of identity (cf. \citealp[Sec. 7.6]{Makarov:2013aa}). A typical construction
of approximate identities is to consider the sequence of dilations,
of the form: $g_{k}\left(\cdot\right)=k^{n}g\left(k\times\cdot\right)$,
which permits the following corollary.
\begin{cor}
\label{cor dilation}Let $g$ be a PDF. Then, the sequence $\left\{ g_{k}:g_{k}\left(\cdot\right)=k^{n}g\left(k\times\cdot\right)\right\} $
satisfies the hypothesis of Lemma \ref{lem approximate identities}
and hence permits its conclusion.
\end{cor}
\begin{lem}
\label{lem NVS}Let $\left(\mathbb{E},\left\Vert \cdot\right\Vert _{\mathbb{E}}\right)$
be an NVS, and let $\mathbb{F}\subset\mathbb{E}$ and $u\in\mathbb{E}$.
Then the following statements are equivalent: (a) $u\in\overline{\mathbb{F}}$;
(b) $\mathbb{B}\left(u,r\right)\cap\mathbb{F}\ne\emptyset$, for all
$r>0$; and (c) there exists a sequence $\left\{ u_{m}\right\} \subset\mathbb{F}$
that converges to $u$.
\end{lem}
Let $\mathbb{E}$ be a locally convex linear topological space over
$\mathbb{R}$ and recall that a functional is a function defined on
$\mathbb{E}$ (or some subspace of $\mathbb{E}$), with values in
$\mathbb{R}$. We denote the due space of $\mathbb{E}$ (the space
of all continuous linear functions on $\mathbb{E}$) by $\mathbb{E}^{*}$.
\begin{lem}
[Second geometric form of the Hahn-Banach theorem]\label{lem hahn banach}Let
$\mathbb{A},\mathbb{B}\subset\mathbb{E}$ be two nonempty convex subsets,
such that $\mathbb{A}\cap\mathbb{B}\ne\emptyset$. Assume that $\mathbb{A}$
is closed and that $\mathbb{B}$ is compact. Then, there exists a
continuous linear functional $\phi\in\mathbb{E}^{*}$, such that its
corresponding hyperplane $H=\left\{ u\in\mathbb{E}:\phi\left(u\right)=\alpha\right\} $
($\alpha\in\mathbb{R}$) strictly separates $\mathbb{A}$ and $\mathbb{B}$.
That is, there exists some $\epsilon>0$, such that $\phi\left(u\right)\le\alpha-\epsilon$
and $\phi\left(v\right)\ge\alpha+\epsilon$, for all $u\in\mathbb{A}$
and $v\in\mathbb{B}$. Or, in other words, $\sup_{u\in\mathbb{A}}\phi\left(u\right)<\inf_{v\in\mathbb{B}}\phi\left(v\right)$. 
\end{lem}
\begin{lem}
[Riesz representation theorem for $\mathcal{L}_p$, $p\in\mathbb{R}_{+}$]\label{lem Riesz p}If
$p\in\mathbb{R}_{+}$, and $\phi\in\left(\mathcal{L}_{p}\right)^{*}$,
then, there exists a unique function $u\in\mathcal{L}_{q}$, such
that for all $v\in\mathcal{L}_{q}$,
\[
\phi\left(v\right)=\int u\left(x\right)v\left(x\right)\text{d}\lambda\left(x\right)\text{,}
\]
where $1/p+1/q=1$.
\end{lem}
\begin{lem}
[Riesz representation theorem for $\mathcal{L}_1$]\label{lem Riesz 1}If
$\phi\in\left(\mathcal{L}_{1}\right)^{*}$, then there exists a unique
$u\in\mathcal{L}_{\infty}$, such that for all $v\in\mathcal{L}_{1}$,
\[
\phi\left(v\right)=\int u\left(x\right)v\left(x\right)\text{d}\lambda\left(x\right)\text{.}
\]
\end{lem}
\begin{lem}
\label{lem part unity}Let $\mathbb{V}_{1},\dots,\mathbb{V}_{n}$
be open subsets of $\mathbb{R}^{n}$, and let $\mathbb{K}$ be a compact
set, such that $\mathbb{K}\subset\bigcup_{i=1}^{n}\mathbb{V}_{i}$.
Then, there exists functions $h_{i}\prec\mathbb{V}_{i}$ ($i\in\left[n\right]$),
such that $\sum_{i=1}^{n}h_{i}\left(x\right)=1$, for all $x\in\mathbb{K}$.
The set $\left\{ h_{i}\right\} $ is referred to as the partition
of unity on $\mathbb{K}$, subordinated to the cover $\left\{ \mathbb{V}_{i}\right\} $.
\end{lem}
\begin{lem}
\label{lem ball cover}If $\mathbb{X}\subset\mathbb{R}^{n}$ is bounded,
then for any $r>0$, $\mathbb{X}$ can be covered by $\bigcup_{i=1}^{m}\mathbb{B}\left(x_{i},r\right)$,
for some finite $m\in\mathbb{N}$, where $x_{i}\in\mathbb{R}^{n}$
and $i\in\left[m\right]$.
\end{lem}
\begin{lem}
\label{lem inclusion}If $1\le p\le q\le r\le\infty$, then $\mathcal{L}_{p}\cap\mathcal{L}_{r}\subset\mathcal{L}_{q}$.
\end{lem}
\begin{lem}
\label{lem convolution}If $f\in\mathcal{L}_{p}$ ($1\le p\le\infty$)
and $g\in\mathcal{L}_{1}$, then $f\star g$ exists and we have $\left\Vert f\star g\right\Vert _{\mathcal{L}_{p}}\le\left\Vert f\right\Vert _{\mathcal{L}_{p}}\left\Vert f\right\Vert _{\mathcal{L}_{1}}$.
Furthermore, if $p$ and $q$ are such that $1/p+1/q=1$, then $f\in\mathcal{L}_{p}$
and $g\in\mathcal{L}_{q}$, then $f\star g$ exists, is bounded and
uniformly continuous, and $\left\Vert f\star g\right\Vert _{\mathcal{L}_{\infty}}\le\left\Vert f\right\Vert _{\mathcal{L}_{p}}\left\Vert f\right\Vert _{\mathcal{L}_{q}}$.
In particular, if $p\in\mathbb{R}_{+}$, then $f\star g\in\mathcal{C}_{0}$.
\end{lem}
\begin{lem}
[Fubini's Theorem]\label{lem fubini}Let $\left(\mathbb{X},\mathcal{X},\nu_{1}\right)$
and $\left(\mathbb{Y},\mathcal{Y},\nu_{2}\right)$ be $\sigma\text{-finite}$
measure spaces, and assume that $f$ is a $\left(\mathcal{X}\times\mathcal{Y}\right)\text{-measurable}$
function on $\mathbb{X}\times\mathbb{Y}$. If
\[
\int_{\mathbb{X}}\left[\int_{\mathbb{Y}}\left|f\left(x,y\right)\right|\text{d}\nu_{1}\left(x\right)\right]\text{d}\nu_{2}\left(y\right)<\infty\text{,}
\]
then 
\begin{align*}
\int_{\mathbb{X}\times\mathbb{Y}}\left|f\right|\text{d}\left(\nu_{1}\times\nu_{2}\right) & =\int_{\mathbb{X}}\left[\int_{\mathbb{Y}}\left|f\left(x,y\right)\right|\text{d}\nu_{1}\left(x\right)\right]\text{d}\nu_{2}\left(y\right) =\int_{\mathbb{Y}}\left[\int_{\mathbb{X}}\left|f\left(x,y\right)\right|\text{d}\nu_{2}\left(y\right)\right]\text{d}\nu_{1}\left(x\right)<\infty\text{.}
\end{align*}
\end{lem}

\subsection*{Sources for results}

Lemma \ref{lem approximate identities} appears in \citet[Thm. 9.3.3]{Makarov:2013aa}
and \citet[Ch. 20, Thm. 2]{Cheney:2000aa}. Corollary \ref{cor dilation}
is obtained from \citet[Ch. 20, Thm. 4]{Cheney:2000aa}. Lemmas \ref{lem NVS},
\ref{lem inclusion}, and \ref{lem convolution} are taken from Propositions
0.22, 6.10, and 8.8 \citet{Folland:1999aa}. Lemmas \ref{lem hahn banach}--\ref{lem Riesz 1}
appear in \citet{Brezis:2010aa} as Theorems 1.7, 4.11, and 4.14,
respectively. Lemmas \ref{lem part unity} and \ref{lem fubini} can
be found in \citet{Rudin:1987aa} as Theorems 2.13 and Theorem 8.8,
respectively. Lemma \ref{lem ball cover} is obtained from \citet[Thm. 1.2.2]{Conway:2012aa}.

\section{Acknowledgements}
The authors would like to very much thank Pr. Eric Ricard for the interesting discussions with him and for his suggestions. TTN is supported by ``Contrat doctoral" from the French Ministry of Higher Education and Research and by the French National Research Agency (ANR) grant SMILES ANR-18-CE40-0014. HDN and GJM are funded by Australian Research Council grant number DP180101192.

\bibliographystyle{apalike}
\bibliography{Bibliography}

\end{document}